\documentclass[12pt,a4paper,leqno]{amsart}

\usepackage{amssymb,hyperref}
\usepackage{array,float}
\usepackage{amsmath,amsthm,amsbsy,amscd,mathrsfs}

\makeatletter
\newcommand\suchthat{%
 \@ifstar
  {\mathrel{}\middle|\mathrel{}}
  {\mid}%
}
\makeatother

\def\bb{\mathbb}

\def\cal{\mathcal}

\def\ZZ{\bb{Z}}
\def\NN{\bb{N}}

\def\NN{\bb{N}}
\def\ZZ{\bb{Z}}

\def\rm{\textrm}
\def\PSL{\rm{PSL}}
\def\soc{\rm{soc}}
\def\Frat{\rm{Frat}}

\newtheorem{theorem}{Theorem} 
\newtheorem{lemma}[theorem]{Lemma}
\newtheorem{proposition}[theorem]{Proposition}
\newtheorem{corollary}[theorem]{Corollary}

\theoremstyle{definition}
\newtheorem{definition}[theorem]{Definition}

\title[]{Recognizing $\PSL(2,p)$ in the non-Frattini chief factors
of finite groups}

\author[Duong Hoang Dung]{Duong Hoang Dung} 
\address{Fakult\"{a}t f\"{u}r Mathematik,
Universit\"{a}t Bielefeld,
Postfach 100131,
D-33501 Bielefeld, Germany.}
\email{dhoang@math.uni-bielefeld.de}

\keywords{Finite groups; Probabilistic zeta function}

\subjclass[2010]{20D06}

\begin{document}

\begin{abstract}
Given a finite group $G$, let $P_G(s)$ be the probability that $s$ randomly chosen
elements generate $G$, and let $H$ be a finite group with $P_G(s)=P_H(s)$. 
We show that if the nonabelian composition factors of $G$ and $H$ are
 $\PSL(2,p)$ for some non-Mersense prime $p\geq 5$, then $G$ and $H$ have the same non-Frattini chief factors.
\end{abstract}

\maketitle

\section{Introduction}
Let $G$ be a finite group. The probability $P_G(s)$
that $s$ randomly chosen elements generate $G$ is calculated as follows (\cite{Hall36}):
\begin{equation}\label{eq-PG}
P_G(s)=\sum_{n\geq 1}\frac{a_n(G)}{n^s},~\rm{where}~a_n(G)=\sum_{|G:H|=n}\mu_G(H).
\end{equation}
Here $\mu_G$ is the M\"obius function on the subgroup lattice
of $G$ defined recursively by $\mu_G(G)=1$ and $\mu_G(H)=-\sum_{H<K\leq G}\mu_G(K)$
if $H<G$. Considering \eqref{eq-PG} as a formal Dirichlet series associated
to $G$, if $G=\ZZ$ then 
$$P_{\ZZ}(s)=\sum_{n\geq 1}\frac{\mu(n)}{n^s}=\frac{1}{\zeta(s)},$$
where $\mu$ is the usual number-theoretic M\"{o}bius function and $\zeta(s)$
is the Riemann zeta function.
The inverse of $P_G(s)$ is then called
the \textit{probabilistic zeta function} of $G$; see \cite{Boston} and \cite{Mann96}. 

Note that if $\mu_G(H)\ne 0$ then $H$ is an intersection of maximal subgroups of $G$, cf. \cite{Hall36}. This implies $P_G(s)=P_{G/\Frat(G)}(s)$, where $\Frat(G)$ denotes the Frattini subgroup of $G$ - the intersection of the maximal subgroups of $G$. Hence, one can only hope to get back information of
$G/\Frat(G)$ from the knowledge of $P_G(s)$.

One natural question asks what we can say about $G$ and $H$ whenever
$P_G(s)=P_H(s)$. It's known that if $G$ is a simple group, then $H/\Frat(H)\cong G$, cf.~{\cite{Lu-Simple, Pat-Algebra}}. When $G$ is not simple, the problem
becomes much harder. Patassini makes a significant progress by obtaining the following results.
\begin{theorem}{\cite{Pat-Comm}}\label{abelian}
Let $G$ and $H$ be finite groups with $P_G(s)=P_H(s)$. Then $G$ and $H$
have the same non-Frattini abelian chief factors.
\end{theorem} 

\begin{theorem}{\cite[Theorem 3]{Pat-Trans}}\label{alternating}
Let $G$ and $H$ be finite groups whose nonabelian composition factors 
are alternating groups $\rm{Alt}(k)$ where either $5\leq k\leq 4.2\cdot 10^{16}$
or $k\geq (e^{e^{15}}+2)^3$. If $P_G(s)=P_H(s)$ then $G$ and $H$ have the same non-Frattini chief factors.
\end{theorem}

Using the same method,  we prove in this paper the following.

\begin{theorem}\label{main theorem}
Let $G$ and $H$ be  finite groups such that $P_G(s)=P_H(s)$. Assume
that the nonabelian composition factors of $G$ and $H$ are  $\PSL(2,p)$, 
for some non-Mersense prime $p\geq 5$.
Then $G$ and $H$ have the same non-Frattini chief factors.
\end{theorem}
From the proof of Theorem~\ref{main theorem}, one has the following consequence.
\begin{corollary}
Let $G$ and $H$ be finite groups such that $P_G(s)=P_H(s)$. Assume
that the nonabelian composition factors of $G$ and $H$ are either $\PSL(2,p)$, 
for some non-Mersense prime $p\geq 5$ or alternating 
groups $\rm{Alt}(k)$ with $k$ satisfying the hypothesis of Theorem~\ref{alternating}.
Then $G$ and $H$ have the same non-Frattini chief factors.
\end{corollary}

In order to mimic the method from \cite{Pat-Trans}, we  prove the following result.
\begin{theorem}\label{theorem irred}
Let $L$ be a monolithic primitive group with socle $\soc(L)\cong\PSL(2,p)^n$ with $p\geq 5$  
a non-Mersense prime. Then the Dirichlet polynomial 
$$P_{L,\soc(L)}(s)=\sum_{m\in\NN}\frac{b_m({L,\soc(L)})}{m^s},~\rm{with}~b_m({L,\soc(L))}=\sum_{\substack{|L:H|=m\\ H\soc(L)=L}}\mu_L(H)$$
is irreducible in the ring of finite Dirichlet polynomials.
\end{theorem}

\subsection*{Notations}
In this paper, groups are always finite. Given a finite Dirichlet series $F(s)=\sum_{n\in\NN}a_n/n^s$ and
a set of prime numbers $\pi$, we denote by $F^{(\pi)}(s)$ the Dirichlet series obtained from $F(s)$ by deleting the
$a_n/n^s$ with $n$ divisible by a prime in $\pi$. For a positive integer $m$, we denote by $\pi(m)$ the set of prime divisors of $m$.  For a given finite group $H$, we write $\pi(H)$ for $\pi(|H|)$.

\section{Preliminaries}\label{sec 2}
Given a normal subgroup $N$ of $G$, it's shown in {\cite[Section 2.2]{Brown00}}
that 
\begin{equation}\label{eq-Brown-fac.}
P_G(s)=P_{G/N}(s)P_{G,N}(s),
\end{equation}
where
$$P_{G,N}(s)=\sum_{n\in\NN}\frac{b_n(G,N)}{n^s},~\rm{with}~b_n(G,N)=\sum_{\substack{|G:H|=n\\ HN=G}}\mu_G(H).$$
By taking a chief series 
$$\Sigma:1=G_k<\cdots< G_1<G_0=G,$$
and iterating equation \eqref{eq-Brown-fac.}, we can express $P_G(s)$ as a product
of Dirichlet polynomials indexed by the non-Frattini chief factors in $\Sigma$:
\begin{eqnarray}\label{eq-P_G-fac.}
P_G(s)=\prod_{G_i/G_{i+1}\not\leq\Frat(G/G_{i+1})}P_{G/G_{i+1},G_i/G_{i+1}}(s).
\end{eqnarray}
It was proved in \cite{DeLu03} that the factors in
\eqref{eq-P_G-fac.} are independent of the choice of the series $\Sigma$. Moreover,
it also describes how those factors look like as follows.

Let $A$ be a minimal normal subgroup of $G$. The monolithic primitive group associated to $A$
is defined as
$$
L_A:=\left\{\begin{array}{ll}
A\rtimes G/C_G(A) & \rm{if $A$ is abelian},\\
G/C_G(A) & \rm{otherwise}.
\end{array}
\right.$$
Note that $A\cong\soc(L_A)$. Define
\begin{align*}
\widetilde{P}_{L_A,1}(s)&=P_{L_A,A}(s),\\
\widetilde{P}_{L_A,i}(s)&=P_{L_A,A}(s)-\frac{(1+q_A+\cdots+q_A^{i-2})\gamma_A}{|A|^s},~\rm{for}~i>1,
\end{align*}
where $\gamma_A=|C_{\rm{Aut}(A)(L_A/A)}|$ and $q_A=|\rm{End}_{L_A}(A)|$ if $A$ is abelian, $q_A=1$ otherwise.
If $A=H/K$ is a non-Frattini chief factor of $G$ then $P_{G/K,H/K}(s)=\widetilde{P}_{L_A,A}(s)$
where $\widetilde{P}_{L_A,A}(s)$ is one of the $\widetilde{P}_{L_A,i}(s)$ for a suitable choice $i$, cf.~{\cite[Theorem 17]{DeLu03}}.

If $A$ is abelian then
$$P_{L_A,A}(s)=1-\frac{c(A)}{|A|^s},$$
where $c(A)$ is the number of complements of $A$ in $L_A$, cf.~{\cite{Gaschutz59}}. Assume that $A\cong S_A^n$ is nonabelian. Let $X_A$ be the subgroup of $\rm{Aut}(S_A)$ induced by the conjugation
action of the normalizer in $L_A$ of a simple component of $S_A^n$. Then $X_A$ is an almost simple group
with socle $S_A$, cf.{\cite[Section 2]{DaLu07}}.
\begin{proposition}{\cite[Theorem 5]{Seral}}\label{Seral}
$$\widetilde{P}^{(r)}_{L_A,A}(s)=P^{(r)}_{L_A,A}(s)=P^{(r)}_{X_A,S_A}(ns-n+1)$$
for every prime divisor $r$ of the order of $S_A$.
\end{proposition}

\section{Proof of Theorem~\ref{theorem irred}}
In this section, we will prove Theorem~\ref{theorem irred}. It enables
us to obtain the irreducibility of the $\widetilde{P}_{L_A,A}(s)$ in 
the factorization of $P_G(s)$. First of all, we will recall some useful results.

Let $\cal{R}$ be the ring of Dirichlet polynomials with integer coefficients
$$\cal{R}=\left\{\sum_{n\in\NN}\frac{a_n}{n^s}: a_n\in\ZZ,~\{n:a_n\ne 0\}|<\infty \right\}.$$
Given a set of prime numbers $\pi$, denote by $X_\pi$ the set
of commuting indeterminates $\{x_r:r\in\pi\}$. Let $\cal{R}'$ be the subring
of $\cal{R}$ containing polynomials $\sum_{n\in\NN}a_n/n^s$ such that $a_n\in n\ZZ$ for every $n$, and $\cal{R}'_{\pi}$ its subring such that $a_n\ne 0$ 
whenever $n$ is a $\pi'$ number. The rings $\cal{R}, \cal{R}'$ and $\cal{R}'_\pi$ are factorial rings, cf.\cite{DaLuMo04}. There is an isomorphism
$\Phi$ between $\cal{R}'_\pi$ and the polynomial ring $\ZZ[X_\pi]$ mapping $p^{1-s}$
to $x_p$ for every $p\in\pi$.
\begin{lemma}\label{irred test 1}{\cite[Lemma 10]{Pat-Israel}}
Let $D$ be a factorial domain. If the polynomial $f(x)=1-ax^m\in D[x]$ is
reducible then $a$ or $-a$ is a non-trivial power in $D$.
\end{lemma}

For a Dirichlet polynomial $F(s)=\sum_{n\in\NN}a_n/n^s\in\cal{R}$ and $v$ a 
prime number, denote by $|F(s)|_v$ the maximal $v$-part of $n$ such that $a_n\ne 0$.

\begin{lemma}\label{irred test 2}{\cite[Lemma 12]{Pat-Israel}}
Let $h(s)=\sum_{k\in\NN}$ be a Dirichlet polynomial and let $m$ be the
least common multiple of $\{k:a_k\ne 0\}$. Assume the following hold:
\begin{itemize}
\item There exists a set of prime number $\pi_0$ such that $h^{(\pi_0)}(s)$
is irreducible.
\item There exists a non-empty subset $\pi$ of $\pi(m)$ such that 
$|h^{(\pi_0)}(s)|_v=|m|_v$ for all $v\in\pi$.
\end{itemize}
Then $h(s)$ is irreducible in $\cal{R}$ if and only if $(h(s),h^{(\pi)}(s))=1$.
\end{lemma}
\begin{definition}
Let $n\in\NN_{>1}$. A  prime number $p$ is called a \textit{primitive prime
divisor} of $a^n-1$ if it divides $a^n-1$ but does not divide $a^e-1$ for 
any integer $1\leq e\leq n-1$.
\end{definition}
\begin{proposition}{\cite{Zsigmondy}}
Let $a$ and $n$ be integers greater than $1$. There exists a primitive prime
divisor of $a^n-1$ except exactly in the following cases:
\begin{itemize}
\item $n=2, a=2^s-1$, where $s\geq 2$.
\item $n=6, a=2$.
\end{itemize}
\end{proposition}
Notice that there may be more than one primitive prime divisor of $a^n-1$. Such
a prime is called a \textit{Zsigmondy prime} for $\langle a,n\rangle$.

\begin{proof}[Proof of Theorem~\ref{theorem irred}]
If $p=5$ then $\PSL(2,5)\cong\rm{Alt}(5)$ and the result follows from {\cite[Theorem~2]{Pat-Trans}}.
Assume now that $p>5$ and set $S:=\PSL(2,p)$.
Let $t$ be the largest prime divisor of $(p-1)$ such that $t$ does not divide $(p+1)$. Set $\pi_0=\{t\}$. Since $p$ is not a Mersense prime, $\pi_0$ is non-empty. Let $r$ be a Zsigmondy prime for $\langle p,2\rangle$ and denote by $x=x_r$ the indeterminate corresponding to the
prime $r$. Let $D=\ZZ[X_{\pi(S)\setminus\{r\}}]$.

We first claim that $P^{(t)}_{L,\soc(L)}(s)$ is irreducible. Let
$X:=X_S$ be its associated almost simple group as in Section~\ref{sec 2}. Note that  
$P^{(t)}_{L,\soc(L)}(s)=P^{(t)}_{X,S}(ns-n+1)$; cf.\ Proposition~\ref{Seral}.

\textit{Case (i): $X=S=\PSL(2,p)$}. Then $P^{(t)}_{X,S}(ns-n+1)=P^{(t)}_S(ns-n+1)$.
It follows from {\cite[Section 7]{Pat-Pacific}} that
$$f(x)=\Phi(P^{(t)}_{L,\soc(L)}(s))=1-ax^m,$$
for some $m\in\NN$ and $a=bx_p^n+c$ with $b,c\in \ZZ[X_{\pi(S)\setminus\{r,p\}}]$. By inspection, $b$ and $c$ are nonzero and hence $a$ or $-a$
can not be a non-trivial power in $D$. Hence $P^{(t)}_{L,\soc(L)}(s)$
is irreducible by Lemma~\ref{irred test 1}.

\textit{Case (ii): $X=\rm{PGL}(2,p)$}. It follows from {\cite[Lemma 5]{VoLuMo2003}} that
$$P_{X,S}(s)=-\sum_{H\leq S}\frac{\mu_X(H)}{|S:H|^s}.$$
The list of subgroups $H$ and $\mu_X(H)$
are described in {\cite[Section~3]{VoLuMo2003}}. By inspection and argument as
in Case~(i), we obtain
the irreducibility of $P_{L,\soc(L)}^{(t)}(s)=P^{(t)}_{X,S}(ns-n+1)$.

Let $\pi=\{p,r\}$. One has that $|P_{L,\soc(L)}^{(t)}(s)|_p=|\PSL(2,p)|_p^n$ and
$|P_{L,\soc(L)}^{(t)}|_r=|\PSL(2,p)|_r^n$. Moreover $P_{L,\soc(L)}^{(\pi\cup\{t\})}(s)=1$. Hence it follows from Lemma~\ref{irred test 2} that
$P_{L,\soc(L)}(s)$ is irreducible.
\end{proof}

\begin{corollary}\label{cor irred}
Let $L$ be a monolithic primitive group with socle $\soc(L)\cong\PSL(2,p)^n$,
where $p\geq 5$ is a non-Mersense prime. Then the Dirichlet polynomial $\widetilde{P}_{L,\soc(L)}(s)$ is irreducible.
\end{corollary}

\section{Proof of Theorem~\ref{main theorem}}

Analogous to {\cite[Proposition 28]{Pat-Trans}}, we have the following crucial result.

\begin{proposition}\label{prop: crucial}
Let $X$ be an almost simple group with socle $\PSL(2,p)$, and
$Z$ an almost simple group such that $P_{X,\PSL(2,p)}^{(r)}=
P_{Z,\soc(Z)}^{(r)}(s)$ for every prime $r\leq p$. Then $\soc(Z)\cong\PSL(2,p)$.
\end{proposition}
\begin{proof} 
Let $n$ be the minimal index of a proper subgroup of $X$ which supplements $\PSL(2,p)$. 

If $p=5$ then $n=5$ and  $\PSL(2,5)\cong\rm{Alt}(5)$. The result follows from {\cite[Proposition 28]{Pat-Trans}}. 
If $p=11$ then $X=\PSL(2,11)$, $n=11$ and (cf. {\cite[Section 7]{Pat-Pacific}})
$$P_{\PSL(2,11)}(s)=1-\frac{22}{11^s}-\frac{12}{12^s}+\frac{66}{66^s}+\frac{220}{110^s}
+\frac{132}{132^s}+\frac{165}{165^s}-\frac{220}{220^s}-\frac{990}{330^s}+\frac{660}{660^s}.$$
By GAP~\cite{GAP4}, the possibility for $Z$ is that $Z=M_{11}$. However by considering  $P^{(2)}_{\PSL(2,11)}(s)=P^{(2)}_{Z,\soc(Z)}(s)$ and noting that $M_{11}$ has a maximal subgroup of index $55$ (cf.~\cite{Atlas}), we obtain a contradiction. This implies
$\soc(Z)\cong\PSL(2,11)$. 

Assume now that $p>11$. Then $n=p+1$. 
Since $P^{(p)}_{X,\PSL(2,p)}(s)=P^{(p)}_{Z,\soc(Z)}(s)$ and $a_{p+1}(X,\PSL(2,p))\ne0$; see {\cite[Section 7]{Pat-Pacific}} and
{\cite[Section 4]{VoLuMo2003}}, it follows 
that $a_{p+1}(X,\PSL(2,p))=a_{p+1}(Z,\soc(Z))\ne0$. Thus the minimal index $k$ of a proper subgroup of $Z$ is at most $p+1$. For a contradiction, assume that $k<p+1$. Since $P^{(p)}_{X,\PSL(2,p)}(s)=P^{(p)}_{Z,\soc(Z)}(s)$,  $k$ should be divisible by $k$, so $k=p$. By considering  $P^{(2)}_{X,\PSL(2,p)}(s)=P^{(2)}_{Z,\soc(Z)}(s)$, since $a_p(X,\PSL(2,p))=0$ (see {\cite[Section 7]{Pat-Pacific}} and
{\cite[Section 4]{VoLuMo2003}}),  we get $a_p(Z,\soc(Z))=a_p(X,\PSL(2,p))=0$, a contradiction. Hence $k=p+1$. Since $Z$ is  almost simple, it is a primitive group of degree $p+1$. 

Since $P^{(2)}_{X,\PSL(2,p)}(s)=P^{(2)}_{Z,\soc(Z)}(s)$,  {\cite[Section 7]{Pat-Pacific}} and
{\cite[Section 4]{VoLuMo2003}} implies
that $p$ divides $|Z|$. Thus $Z$ contains a $p$-cycle, since $Z$
is a primitive group of degree $p+1$ and $p>(p+1)/2$.
It follows from \cite{Chang} that either $Z=M_{24}$ or $\soc(Z)=\rm{Alt}(p+1)$ or $\soc(Z)=\PSL(2,p)$.
If $Z=M_{24}$, then by considering $P^{(2)}_{X,\PSL(2,23)}(s)=P^{(2)}_{Z,\soc(Z)}(s)$ and noting that $M_{24}$ has
a maximal subgroup of index $7\cdot11\cdot23$, one gets a contradiction.
Assume that $\soc(Z)=\rm{Alt}(p+1)$. 
By Bertrand's postulate (cf. \cite{Ramanujan}),  there exists a prime $l$ such that $(p+1)/2<l<p-1$. By considering $P^{(p)}_{X,\PSL(2,p)}(s)=P^{(p)}_{Z,\soc(Z)}(s)$ one obtains
from {\cite[Theorem 1.1]{DaLu07}} that $l$  divides $|X|$, which is a contradiction. As a conclusion, we have that $\soc(Z)\cong\PSL(2,p)$.
\end{proof}

\begin{proof}[Proof of Theorem~\ref{main theorem}]
The proof is analogous to that of {\cite[Theorem 3]{Pat-Trans}}. We present here for the sake of completeness of the paper.

By Theorem~\ref{abelian},  $G$ and $H$ have the same  non-Frattini abelian chief factors. 
Thus we may assume that $G$ and $H$ have no non-Frattini abelian chief factors.

Let $\cal{CF}(G)$ be the set of the non-Frattini chief factors of $G$.
For each $A\in\cal{CF}(G)$, the polynomial $\widetilde{P}_{L,A}(s)$
is irreducible; cf.~Corollary~\ref{cor irred}. Hence
$$P_G(s)=\prod_{A\in\cal{CF}(G)}\widetilde{P}_{L,A}(s)$$
is a factorization of $P_G(s)$ into irreducible factors.
Thus, there is a bijection between the sets $\cal{CF}(G)$ and $\cal{CF}(H)$ such that $A\cong S_A^{n_A}\in\cal{CF}(G)$ and 
$B\cong S_B^{n_B}\in\cal{CF}(H)$ are associated if and only if $\widetilde{P}_{L_A,A}(s)=\widetilde{P}_{L_B,B}(s)$.
Since $\widetilde{P}^{(r)}_{L_A,A}(s)=\widetilde{P}^{(r)}_{L_B,B}(s)$ for every $r\in\pi(A)$. It follows that $n_A=n_B$, cf.~{\cite[Proposition~27]{Pat-Trans}}. Thus $\widetilde{P}^{(r)}_{X_A,S_A}(s)=\widetilde{P}^{(r)}_{X_B,S_B}(s)$ for every $r\in\pi(A)$, cf. Proposition~\ref{Seral}.
Proposition~\ref{prop: crucial} implies that $S_A\cong S_B$. Therefore $A\cong S_A^{n_A}\cong S_B^{n_B}\cong B$ as desired.
\end{proof}


\subsection*{Acknowledgment}
This research is supported by the DFG Sonderforschungsbereich 701 at
Bielefeld University.
We acknowledge the referee for valuable comments. We also thank Gareth Jones and Massimiliano Patassini for useful discussions.

 -------------------------------------------------------------------
\end{document}